\documentclass[12pt]{amsart}

\usepackage{amsmath,amssymb}
\usepackage{datetime}
\usepackage{graphicx}
\usepackage[utf8]{inputenc}
\usepackage{color}
\usepackage{xcolor}

\DeclareMathOperator*{\vol}{vol}

\newtheorem{thm}{Theorem}[section]
\newtheorem{lem}[thm]{Lemma}
\newtheorem{prop}[thm]{Proposition}
\newtheorem*{rem}{Remark}
\newtheorem{definition}[thm]{Definition}

\newcommand{\Z}{{\mathbb Z}}
\newcommand{\N}{{\mathbb N}}

\newcommand{\R}{{\mathbb R}}
\newcommand{\tmop}[1]{\ensuremath{\operatorname{#1}}}
\def\tmod{\,(\tmop{mod}}
\def\pamod{\! \! \! \! \pmod}

\def \sl  {{\hbox{SL}_2( {\mathbb R})} }

\def \psl  {{\hbox{PSL}_2( {\mathbb R})} }

\def \pslz  {{\hbox{PSL}_2( {\mathbb Z})} }
\newcommand{\Nc}{{\mathcal{N}}}
\newcommand{\Hb}{{\mathbb{H}}}
\newcommand{\Db}{{\mathbb{D}}}
\newcommand{\Sc}{{\mathcal{S}}}
\newcommand{\C}{{\mathbb{C}}}

\numberwithin{equation}{section}

\newcommand{\fixmehidden}[1]{}

\title[Distribution of lattice points on hyperbolic circles]{On the distribution of lattice points on hyperbolic circles}

\author{Dimitrios Chatzakos, P\"ar Kurlberg, Stephen Lester and Igor Wigman}
 \address{IMB, Université de Bordeaux, Bâtiment A33, 33405 Talence, France}
 \email{dimitrios.chatzakos@math.u-bordeaux.fr}
 \address{Department of Mathematics, KTH, 
 SE-100 44 Stockholm, Sweden}
 \email{kurlberg@math.kth.se}
\address{Department of Mathematics, King's College London, London WC2R 2LS, UK}
 \email{steve.lester@kcl.ac.uk}
\address{Department of Mathematics, King's College London, London WC2R 2LS, UK}
\email{igor.wigman@kcl.ac.uk}

\begin{document}
\begin{abstract}
  We study the fine distribution of lattice points lying on expanding
  circles in the hyperbolic plane $\mathbb{H}$. The angles of lattice
  points arising from the orbit of the modular group $\pslz$, and
  lying on hyperbolic circles, are shown to be equidistributed for
  {\em generic} radii. However, the angles fail to equidistribute on a thin
  set of exceptional radii, even in the presence of growing
  multiplicity. Surprisingly, the distribution of angles on hyperbolic
  circles turns out to be related to the angular distribution of
  $\mathbb{Z}^2$-lattice points (with certain parity conditions) lying
  on circles in $\mathbb{R}^2$, along a thin subsequence of radii. A
  notable difference is that measures in the hyperbolic setting can
  break symmetry --- on very thin subsequences they are not invariant
  under rotation by $\frac{\pi}{2}$, unlike the Euclidean setting where
  all measures have this invariance property.

\end{abstract}
\date{\today}

\maketitle

\section{Introduction}
\label{sec:introduction}

\subsection{Background and motivation} We study the angular distribution of lattice points  on
hyperbolic circles --- the boundary of balls with respect to the
hyperbolic distance --- and show that equidistribution holds ``generically''
as the radius grows. We  also show that there are subsequences
where equidistribution fails to hold, even if the multiplicity is
growing. Refined equidistribution results of this style have been studied for
the case of Euclidean circles in \cite{kurlbergwigman}.
For the case of the $n$-dimensional hyperbolic space $\mathbb{H}^n$ equidistribution
results for large annuli of fixed width were studied in
\cite{boca,good,marklofvinogradov,risagerrudnick,risagertruelsen,truelsen}.

In this paper we focus on the $2$-dimensional case. Let $\mathbb{H} := \mathbb{H}^2$ denote the hyperbolic plane
which can be identified with the upper half plane
\begin{align*}
\mathbb{H} = \left\{ x+iy: x \in \mathbb{R}, y \in \mathbb{R}_{>0} \right\}.
\end{align*}
The plane $\mathbb{H}$ is equipped with the hyperbolic distance $\rho(\cdot,\cdot)$, that for $z,w \in \mathbb{H}$ is given by the relation
\begin{align} \label{distancefunction}
\cosh (\rho (z,w)) = 1+ \frac{|z-w|^2}{2 \Im(z) \Im(w)}.
\end{align}
The projective special linear group $\psl := \sl / \{\pm I\}$ acts on
$\mathbb{H}$ by M\"obius transformations: for
\begin{align*}
\gamma =
\begin{pmatrix} a & b \\ c & d  \end{pmatrix}
\in \psl
\end{align*}
and $z \in \mathbb{H}$ fixed, the standard action is given by
\begin{align*}
\gamma(z) = \frac{az+b}{cz+d}.
\end{align*}
In fact, $\psl$ is precisely the group of orientation preserving isometries of $\mathbb{H}$.

We will consider discrete subsets of $\mathbb{H}$ given by the orbit of
various  subsets of the \textit{modular group} $\Gamma = \pslz$.
(The case of congruence subgroups introduces some interesting novel features, and will be addressed in future work).
With $w=i$ and $z = \gamma(i)$ we have
\begin{align} \label{eq:distance}
2 \cosh (\rho( \gamma(i), i)) =
\| \gamma \|^{2} = a^{2}+b^{2}+c^{2}+d^{2} .
\end{align}
For $n \in \N$ let $\Gamma^{n} := \{ \gamma \in \Gamma :
\|\gamma \|^{2} = n \}$, and define the set
\begin{equation}
\label{definesetN}
\mathcal{N} := \{ n \in \mathbb N : |\Gamma^{n}|>0 \}
= \left\{ n=  a^{2}+b^{2}+c^{2}+d^{2} \in \mathbb{N}:
\begin{pmatrix} a & b \\ c & d  \end{pmatrix}
\in \Gamma  \right\}.
\end{equation}

We wish to determine the distribution of the points
$
\{ \gamma(i) : \gamma \in \Gamma^n \}
$
as $n$ grows along integers $n \in \mathcal{N}$.
It is convenient to conformally map
$\mathbb{H}$ into the hyperbolic disc $\mathbb{D} = \{ z \in \mathbb{C}, |z| <1\}$ (endowed with the hyperbolic metric) by
\begin{align} \label{conformalmap}
f(z) =
\frac{z-i}{1-iz}.
\end{align}
Clearly $f(i)=0$, and the points $f(\gamma(i))$, for
$\gamma \in \Gamma^n$ all lie on a circle centered at $0$.
Thus, to determine the distribution of lattice points
on the original hyperbolic circle it suffices to determine
the distribution of
angles (or complex arguments)
$$
\theta(\gamma) :=
\arg f( \gamma(i))
$$
as $\gamma$ ranges over elements in $\Gamma^n$.
In order to study the set of possible configurations  we
define probability measures $\mu_{n}$ (for  $n \in \mathcal{N}$) on
$S^{1}$, supported on a finite number of points, by letting
\begin{equation}
\label{eq:mun def}
\mu_{n} :=
\frac{1}{|\Gamma^n|}
\sum_{\gamma \in \Gamma^{n}}
\delta_{\theta(\gamma) }.
\end{equation}

\subsection{Statement of the principal results}
\vspace{-2mm}
\subsubsection{Generic equidistribution}
Our first result states that lattice points arising from the action of
the modular group are asymptotically {\em equidistributed} for ``generic''
$n \in \mathcal{N}$, in the sense that $\mu_{n}$ weakly tends to
$\mu_{\text{Haar}}$ for $n$ tending to infinity along a generic
subsequence, where $\mu_{\text{Haar}}$ denotes the Haar measure on
$S^{1}$ normalized to have mass one.
(For an illustration of approximate equidistribution, we plot two
example point configurations in Figure~\ref{fig:equi}).
In fact, stronger than that, we
give a quantitative bound for the discrepancy.
\begin{thm}
\label{thm:equidistributed}
Letting $\mathcal{N}(x) = \{ n \in \mathcal{N} : n \le x \}$ we
have
$\mathcal{N}(x) \asymp x/\log x $. Further, for all but
$o( |\mathcal{N}(x)|)$ integers $n \in \mathcal{N}(x)$, we have $|\Gamma^n|\asymp (\log n)^{\log 2 \pm o(1)}$ and
moreover
\begin{align*}
\sup_{I \subset S^{1}}
\left|
\frac{|\{ \gamma \in \Gamma^n : \theta(\gamma) \in I\}| }{|\Gamma^{n}|}
-\frac{|I|}{2\pi}
\right|
\ll \frac{1}{|\Gamma^n|^{\vartheta-o(1)}},
\end{align*}
where $\vartheta=\log(\pi/2)/\log 2=0.651496129\ldots$.
\end{thm}

\begin{figure}[h]
  \centering
\includegraphics[width=7.5cm]{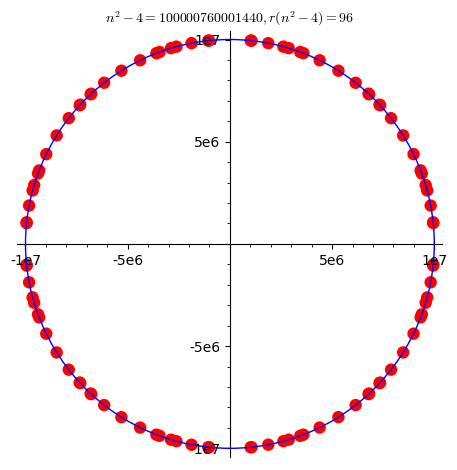}
\includegraphics[width=7.5cm]{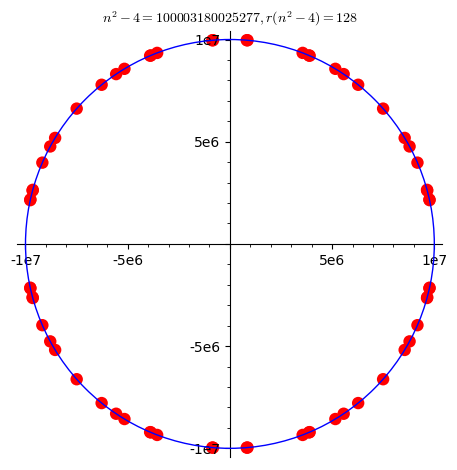}
\caption{Left: Points $(x,y) \in \Z^2$ s.t. $x^{2}+y^{2}=n^{2}-4$. Right:
  Points $(x,y) \in \Z^2$ s.t. $x^{2}+y^{2}=n^{2}-4$ and $x$ even.}
  \label{fig:equi}
\end{figure}

As a corollary to Theorem~\ref{thm:equidistributed}, we determine the
distribution, for generic $n\in\Nc$ tending to infinity,
of the real parts
$
\{ \Re \gamma(i) \tmod 1) : \gamma \in \Gamma^n \},
$
and  show (cf. Section~\ref{sec:distr-real-parts}) that the corresponding
probability density function is given by
\begin{equation}
\label{eq:f pdf Re}
p(x) =
\frac{1}{\pi}
\sum_{k \in \Z} \frac{1}{1+(x+k)^{2}}
=
\frac{\cosh(\pi)\cdot \sinh(\pi)}{\cosh(\pi)^{2}-\cos(\pi x)^{2}},
\quad x\in [0,1].
\end{equation}
That is, for an interval $I\subseteq [0,1]$, the
proportion of those $\gamma\in\Gamma^{n}$ with
$\Re \gamma(i) \tmod 1) \in I$ is asymptotically given by
$\int_{I}p(x)dx$ as $n \rightarrow \infty$ along a density one subsequence of $n \in \mathcal N$.
This can be viewed as a thin set analogue of the
equidistribution result attributed to Good
for the set $\{ \Re \gamma(z) \tmod 1) : \gamma \in \pslz, \Im \gamma(z) \ge \varepsilon \}$, as
$\varepsilon \to 0$ (see \cite[Theorem~6.2]{neunhoffer},
\cite[Eq. (3.27)]{good-equidistribution-real-part-mod-one} and
\cite[Theorem~1.2]{risagerrudnick}).
However, while real parts modulo one are equidistributed when
ordering by the imaginary part (in particular having
constant probability density functions), ordering by distance to $i$
introduces minute fluctuations, cf. Figure~\ref{fig:cosh plot}.
\begin{figure}[h]
  \includegraphics[width=6cm]{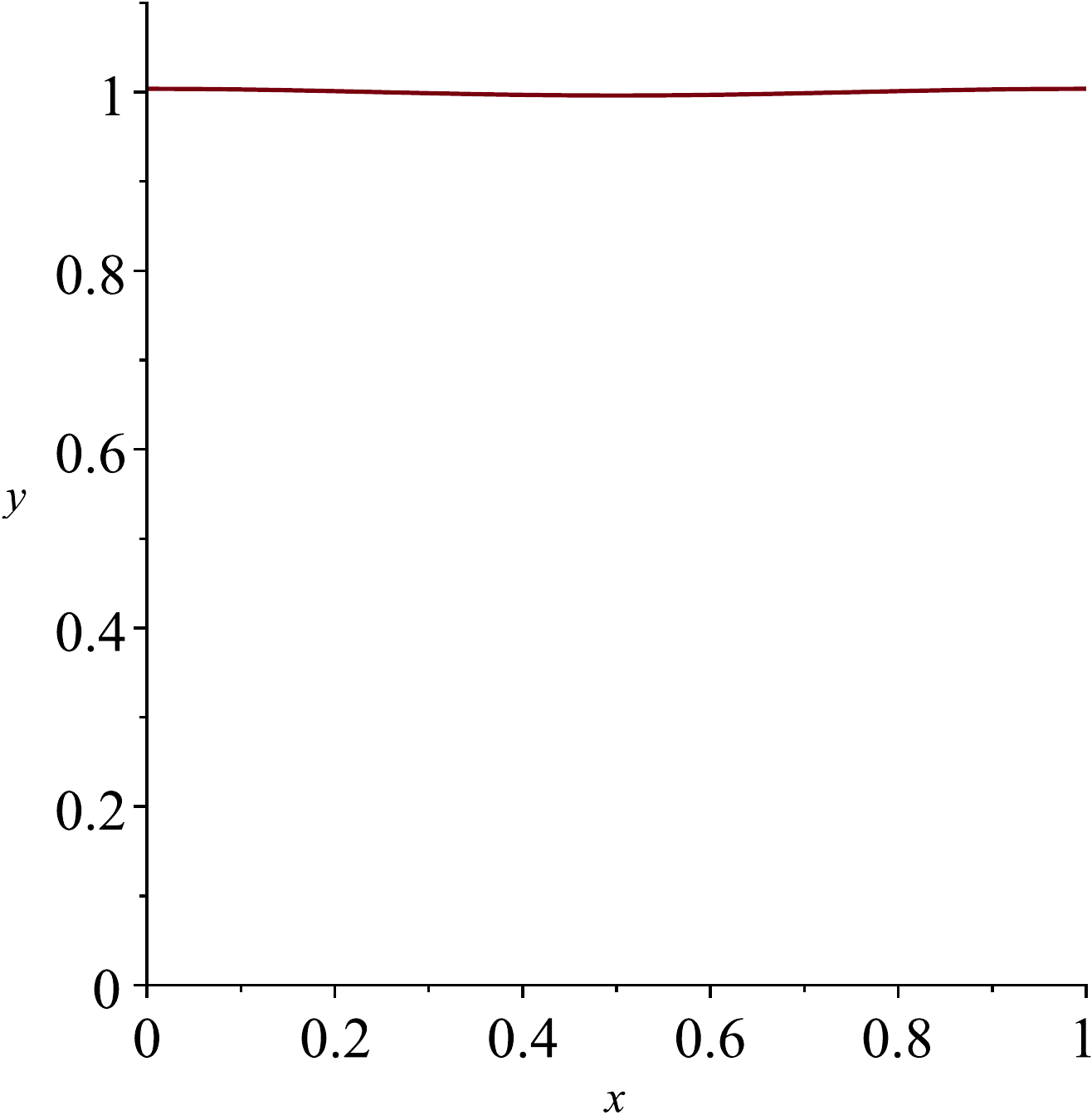}
  \includegraphics[width=6cm]{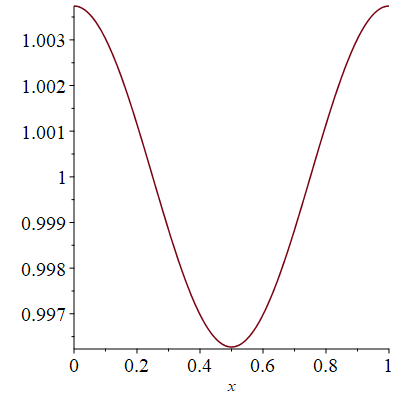}
\caption{Left: The plot of the asymptotic probability density function
$ p(x)= \frac{\cosh(\pi)\cdot \sinh(\pi)}{\cosh(\pi)^{2}-\cos(\pi x)^{2}} $
of $\{ \Re \gamma(i) \tmod 1) : \gamma \in \Gamma^n \}$ along a generic sequence $\{n\}\subseteq \Nc$.
Right: Same plot, magnified in the $y$-direction.}
  \label{fig:cosh plot}
\end{figure}


A key ingredient in the proof of Theorem~\ref{thm:equidistributed} is the
remarkable fact that the set of hyperbolic angles
$\{\theta(\gamma) : \gamma \in \Gamma^{n}\}$ is the {\em same} as the
set of angles of Euclidean $\Z^2$-lattice points, having even
$x$-coordinate, on the circle of radius $\sqrt{n^{2}-4}$.
(In essence, integer points on the {\em surface} given by the two equations
$a^{2}+b^{2}+c^{2}+d^{2} = n$ and $ad-bc=1$ can be identified with
integer points on the {\em curve} given by
$4x^{2}+y^{2} = n^{2}-4$.)

\begin{prop} \label{prop:keyprop}
Let $\Sc:=\{x^{2}+y^{2}:\:x,y\in\Z\}$ denote the set of integers
expressible as sums of two integer squares.  Then
$\mathcal{N} = \{ n \in \Z : n^{2}-4 \in \mathcal{S} \}
$
and, for $n \in \mathcal{N}$, we have
\begin{equation}
  \label{eq:even-x-coordinate-circle}
\{ \theta(\gamma) : \gamma \in \Gamma^{n} \}
=
\left \{ \arg( x + iy) : x,y\in \Z, x \equiv 0 \tmod 2), x^{2} + y^{2} =
  n^{2}-4 \right\}.
\end{equation}
\end{prop}


\subsubsection{Non-equidistribution}

We can also show that there are subsequences $n_{i} \in \mathcal{N}$
tending to infinity in such a way that $|\Gamma^{n_i}| \to \infty$,
yet the angles $\{ \theta(\gamma) : \gamma \in \Gamma^{n_{i}} \}$ fail
to equidistribute. The constraint $|\Gamma^{n_i}| \to \infty$ is
natural, since equidistribution clearly fails along sequences so that
$|\Gamma^{n_{i}}| $ stays bounded (such sequences do exist, see e.g.
\cite[Proposition 2.1]{KLR}).


We will construct a wide family of weak-$*$ partial limits of
the sequence $\{\mu_{n}\}_{n \in \mathcal{N}}$ of probability measures
on $S^{1}$, by comparing the hyperbolic setting to its Euclidean
analogue, i.e. the case of the angular distribution of the points of
$\Z^{2}\subseteq\R^{2}$.
To describe the setup we need some further notation. Given $n \in \Z$, let
$$r(n) := |\{ (x,y)\in \Z^2 : x^{2} + y^{2} = n \}|$$ denote the
number of representations of $n$ as sum of two squares, and given $n
\in \Sc$, define a
probability measure $\nu_n$ on $S^{1}$ by
\begin{equation} \label{eq:nudef}
\nu_{n} :=
\frac{1}{r(n)}
\sum_{(x,y) \in \Z^2 : x^{2}+y^{2}=n}
\delta_{\arg(x+iy)},
\end{equation}
supported precisely on the angles corresponding to these
representations.



We say that a probability measure on $S^{1}$ is attainable from
lattice points on circles, or simply just {\em attainable},
(cf. \cite[Definition 1.1]{kurlbergwigman}) if it is a weak-$*$ partial
limit of $\{\nu_{n}\}_{n\in \Sc}$.
%
Our second principal result asserts that every attainable
measure, under a small perturbation, is a weak-$*$ partial limit of
$\{\mu_{n}\}_{n\in\Nc}$ in \eqref{eq:mun def}.

\begin{thm}
\label{thm:non-equidist}

There exists an absolute constant $C>0$, so that
for every probability measure $\nu$ on $S^{1}$ that is attainable,
there exists a sequence $\{n_{i}\} \subseteq \mathcal{N}$ and a
probability measure $\widetilde{\nu}$ supported on at most $C$
points on $S^{1}$, such that
$$
\mu_{n_i}\Rightarrow\nu \star \tilde{\nu}.
$$
Whether or not $\nu$ is supported on a finite number of points,
we may impose the further condition that
$|\Gamma^{n_{i}}|$ grows as $n_i \rightarrow \infty$.
\end{thm}
In particular there exists a sequence $\{n_{i}\}\subseteq \mathcal{N}$
such that $|\Gamma^{n_i}| \to \infty$, and the weak-$*$ limit
$\mu = \lim_{i \to \infty}\mu_{n_{i}}$ is highly singular in the sense
of having support on a finite number of points, see
Figure~\ref{fig:singular} for illustration.
\begin{figure}[h]
  \centering
\includegraphics[width=8cm]{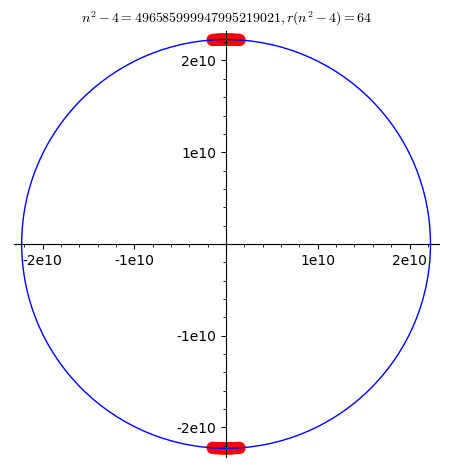}
\caption{Points $(x,y) \in \Z^2$ s.t. $x^{2}+y^{2}=n^{2}-4$ and $x$ even.}
  \label{fig:singular}
\end{figure}

Further, there are
hyperbolic limiting measures $\mu = \nu \star \tilde{\nu}$ that are
singular continuous --- for example, we may take $\nu$ to be a measure
of Cantor type with arbitrary small support
(cf. \cite[\S4.3]{kurlbergwigman}).

\subsubsection{Breaking symmetry}

We remark that for $n \equiv 0 \tmod 4)$ the parity condition
on the $x$-coordinate in \eqref{eq:even-x-coordinate-circle}
is
illusory: {\em both} $x$ and $y$ must be even if
$x^{2}+y^{2} = n^{2}-4 \equiv 0 \tmod 4)$, and in this case the angles
are obtained from {all} lattice points on the circle of radius
$\sqrt{n^{2}-4}$.
%
In particular, any measure \eqref{eq:mun def} (or any weak-$*$ limit
along even $n \in \mathcal{N}$) must be invariant under
$(x,y) \to (x,-y)$ as well as rotation by $\frac{\pi}{2}$.

However,
for $n$ odd, the parity condition breaks the quarter rotation
symmetry, yet the limiting measures along ``generic'' odd $n$'s are
equidistributed by Theorem~\ref{thm:equidistributed} --- hence invariant
under all rotations.  A natural question is whether or not there exist
limit measures of $\mu_{n}$, along subsequences $n_i$ so that $|\Gamma^{n_i}|$ tends to
infinity, that are {\em not} invariant under rotation by
$\frac{\pi}{2}$. This is indeed the case.


\begin{thm} \label{thm:break}
There exists weak-$*$ limit points of $\{\mu_{n}\}_{n \in \mathcal
  N}$, along subsequences so that $|\Gamma^{n}|$ tends to infinity, which
are asymmetric in the sense of not
being invariant under rotation by $\frac{\pi}{2}$.
In particular, the weak-$*$ limit points of
$\{\mu_{n}\}_{n \in \mathcal N}$ do not coincide with those of
$\{\nu_n\}_{n\in\Sc}$.
\end{thm}

\subsection{Comparison with the Euclidean case}
It is natural to compare our results with those for the case of
lattice points of $\mathbb{Z}^2$ inside $\mathbb{R}^2$. In that case, a
more precise description of the (aforementioned ``attainable")
limit measures can be given; in fact,
a partial classification of the first two nontrivial Fourier
coefficients of the limit measures was done by Kurlberg and Wigman
\cite{kurlbergwigman}, and Sartori ~\cite{sartori}, building upon the
pioneering works of Cilleruelo \cite{cilleruelo}, Katai-Környei
\cite{katai-kornyei-lattice-points-on-circles} and Erd{\H o}s-Hall
\cite{erdos-hall}; our discrepancy bound can be viewed as a thin
subset analog, of comparable strength, to the discrepancy bounds in
\cite{erdos-hall,katai-kornyei-lattice-points-on-circles}. In the Euclidean
setting the set of limit measures turns
out to have a surprising {\lq fractal structure\rq}; an analogous statement for the orbit points
of the modular group would follow from plausible conjectures related
to twin primes in arithmetic progressions.

\subsection{Discussion on hyperbolic lattice point counting}

The study of the orbit under the action of Fuchsian groups on the
hyperbolic plane with the use of spectral theory goes back to Delsarte
\cite{delsarte} (also see
\cite[pp. 829-845]{delsarte-ouvres}), Huber \cite{huber},
Selberg \cite{selberg} (with the best error term), and Patterson
\cite{patterson}. Nicholls \cite{nicholls}, using ergodic theory, worked
out the case of the $n$-dimensional space, and G\"unther \cite{gunther},
using spectral theory, generalized Selberg's result for the case of
rank one spaces. Refined results in the hyperbolic lattice point
problem have been extensively studied, such as the angle distribution
\cite{boca,good,marklofvinogradov,risagerrudnick,risagertruelsen,truelsen} and the
pair correlation density \cite{bocaetal,bocaetal2,marklofvinogradov,kelmer,risagersodergren}.
Second moment estimates of the error term and their applications
to the study of quadratic forms and correlation sums of $r(n)$ were addressed by
Chamizo \cite{chamizo1,chamizo2} and Iwaniec \cite{iwaniec}. Further, Friedlander and Iwaniec
\cite{friedlanderiwaniec} studied a modified {\lq hyperbolic\rq} prime
number theorem. Although coming from a different problem, their work
is more {\lq arithmetic\rq} in flavor rather than spectral, and it is
closer to our own approach.

\vspace{2mm}


Our investigation was inspired by the results of Marklof and Vinogradov
\cite{marklofvinogradov}, who resolved the local statistics of lattice points lying
on spherical shells of fixed width, projected on the unit sphere, including the correlation functions
of arbitrary order,
finer than their angular equidistribution due to Nicholls \cite{nicholls}.
Namely, for $\gamma\in \Gamma$ and a fixed point $w \in \mathbb{H}$, denote by
$\phi (\gamma w) =\phi_i (\gamma w) $ the intersection of the unit
hyperbolic circle centered at $i \in \mathbb{H}$ and the semi-infinite
geodesic starting at $i$ and containing $\gamma w$. Fix $s\ge 0$, and, for $t>s$ large, consider the projection
\begin{align*}
\mathcal{P}_w (s, t) = \{\phi (\gamma w): \gamma \in \Gamma / \Gamma_w ,\, t-s \leq \rho(\gamma w, i) <t   \}
\end{align*}
of the emerging spherical shell to the unit circle.
Here $\Gamma_w$ is the stabilizer of the point $w \in \mathbb{H}$,
which is a finite cyclic group \cite[Ch.~2]{iwaniec}. These notions
easily extend to the case of a lattice $\Gamma$ acting
discontinuously on the $n$-dimensional hyperbolic space $\mathbb{H}^n$
for $n \geq 2$.

In this context the main research line has been on understanding the
direction distribution in $\mathcal{P}_w (s, t)$ as $t\rightarrow\infty$, with $s$ \textit{fixed}.
By work of Nicholls \cite{nicholls}, asymptotics for the counting
function is known:
\begin{align*}
\left| \mathcal{P}_w (s, t)\right| \sim c_n \left( 1- e^{-(n-1)s} \right) \cdot e^{(n-1)t},
\end{align*}
for $t \to \infty$ and $s$ fixed (here $c_n$ is an explicit constant
depending on $n$ and $\Gamma$). Further, by
~\cite[Theorem 2]{nicholls} it easily follows that
the angles in $\mathcal{P}_w (s, t)$ asymptotically equidistribute:
for every $\mathcal{A} \subset \mathbb{S}^{n-1}$ with
measure zero boundary and $s$ fixed we have
\begin{align*}
\left| \mathcal{P}_w (s, t) \cap \mathcal{A} \right| \sim \frac{\vol( \mathcal{A})}{\vol(\mathbb{S}^{n-1})} \cdot \left| \mathcal{P}_w (s, t)\right|
\end{align*}
as $t \to \infty$.

\vspace{2mm}

Our method differs from the usual attacks on
hyperbolic lattice counting problem. The equivalence between
hyperbolic angular distribution and Euclidean angular distribution
allows us to use arithmetic tools (sieves) rather than results from
spectral or ergodic theory.

\subsection{Acknowledgments}

D.C. was supported by the Labex CEMPI (ANR-11-LABX-0007-01) and is
currently supported by an IdEx postdoctoral fellowship at IBM,
University of Bordeaux. P.K. was partially supported by the Swedish
Research Council (2016-03701). S.L. is partially supported by EPSRC
Standard Grant EP/T028343/1.  The research leading to these results
has received funding from the European Research Council under the
European Union's Seventh Framework Programme (FP7/2007-2013), ERC
grant agreement n$^{\text{o}}$ 335141 (D.C. and I.W.). We would like
to thank D. R. Heath-Brown and J. Marklof for discussions that
inspired the research presented in this manuscript, and Z. Rudnick for
comments on an earlier version of this manuscript and a very helpful
discussion regarding A. Good's results.

\section{Lattice points in $\mathbb{H}$ and $\R^{2}$: proof of Proposition \ref{prop:keyprop}}
\label{sec:cribro-details}

\subsection{Lattice points on hyperbolic and Euclidean circles}
For $\gamma = \left(\begin{smallmatrix}  a&b\\ c&d \end{smallmatrix} \right)
\in \Gamma$ and $i \in \{1,2,3,4\}$, define
the quadratic
forms $x_{i} = x_{i}(\gamma)$ by
\begin{equation}
  \label{eq:x-forms}
x_{1} = a^{2}+b^{2}, \quad
x_{2} = c^{2}+d^{2}, \quad
x_{3} = ac + bd, \quad
x_{4} = x_1-x_{2};
\end{equation}
we then find that $\gamma(i) = (x_{3}+i)/x_{2} \in \mathbb{H}$, and,
using that $x_3^2 = x_1 x_2 - 1$, a short calculation gives
$$
f(\gamma(i)) = \frac{ 2 x_{3} + i x_{4}}{ x_{1}+x_{2} + 2} \in \mathbb{D}.
$$

Thus, with $n = \| \gamma \|^{2}= a^{2}+b^{2}+c^{2}+d^{2}\in \N$, the set of hyperbolic
angles is given by the angles (in $\R^2$) of the set of points
\begin{align*}
L_{n} :=
\left\{ ( 2 (ac+bd)  , a^{2}+b^{2}-c^{2}-d^{2} ) :
a^{2}+b^{2}+c^{2}+d^{2} = n,  \begin{pmatrix} a & b \\ c & d  \end{pmatrix}
\in \Gamma \right\}.
\end{align*}
We observe that the set $L = \bigcup\limits_{n=1}^{\infty} L_n$ is not a sublattice
of $\mathbb{Z}^2$, and, in what follows, we will show that
$L_{n}$ can be identified with the set
\begin{align*}
\mathcal L_n=
\left \{ (2x_{3}, x_{4}) : x_{3},x_{4}\in \Z, 4x_{3}^{2} + x_{4}^{2} =
  n^{2}-4 \right\}.
\end{align*}

To describe this identification in more detail, let $a,b,c,d \in \Z$
satisfy
\begin{equation*} \notag
a^{2}+b^{2}+c^{2}+d^{2} = n,
\quad ad-bc = 1.
\end{equation*}
Then, following \cite{friedlanderiwaniec} or
\cite[Chapter~14.7]{cribro}\footnote{See
  \cite[Eqs. (14.55) and (14.56)]{cribro}, though there
appears to be a misprint: $y_{2}$ and $y_{4}$ should be interchanged.}
(see also \cite[Ch.12]{iwaniec} and \cite{chamizo1}), we let
\begin{equation}\label{eq:ydef}
y_{1} = a+d, \quad
y_{2} = b-c, \quad
y_{3} = b+c, \quad
y_{4} = a-d.
\end{equation}
It is then simple to verify that
$$
y_{1}^{2} + y_{2}^{2} =
a^{2}+b^{2}+c^{2}+d^{2} + 2(ad-bc)
= n +2
$$
and
$$
y_{3}^{2} + y_{4}^{2} =
a^{2}+b^{2}+c^{2}+d^{2} - 2(ad-bc)
= n - 2,
$$
indicating a correspondence between hyperbolic lattice points, and
Euclidean lattice points on {\em two} circles.
With $\mathcal{N}$
defined as in \eqref{definesetN}, and recalling that
$\mathcal{S} = \{ n \in \mathbb{Z} : r(n) > 0\}$, the above demonstrates that
$\mathcal{N} \subseteq \{ n \in \Z : n \pm 2 \in \mathcal{S} \}$.

\subsection{The parity conditions}\label{sec:parity}
Conversely, we will now show that $\{ n \in \Z : n \pm 2 \in \mathcal{S} \} \subseteq \mathcal{N}$.
This gives rise to certain parity conditions that must be taken into account.

First, note that if $n \in \{ n \in \Z : n \pm 2 \in \mathcal{S} \}$
 then $n+2 \not \equiv 3 \tmod 4) $ (recall that $k \not \in \mathcal{S}$ for
any $k \equiv 3 \tmod 4)$).
Also, if $n+2 \equiv 2 \tmod 4)$, write $n = 4m$;
as $n+2=2(2m+1)$ and $n-2 = 2(2m-1)$ we find that
$n+2 \not \in \mathcal{S} $ or $n-2 \not \in \mathcal{S}$. It remains to consider integers $n$ with $n+2 \equiv 0,1 \tmod 4)$. Let $y_1,y_2,y_3,y_4 \in \mathbb Z$ satisfy
\begin{equation}
\label{eq:y1^2+y2^2=n-2,n+2}
y_1^2+y_2^2=n+2 \qquad \text{ and } \qquad y_3^2+y_4^2=n-2.
\end{equation}
Define
\begin{equation} \label{eq:abcddef}
a = (y_{1}+y_{4})/2, \quad
b = (y_{2}+y_{3})/2, \quad
c = (y_{3}-y_{2})/2, \quad
d = (y_{1}-y_{4})/2.
\end{equation}
Clearly, $a^2+b^2+c^2+d^2=n$ and $ad-bc=1$. We also need that $a,b,c,d \in \mathbb Z$, i.e., that
$y_{1} \equiv y_{4} \tmod 2)$, and $y_{2} \equiv y_{3} \tmod
2)$.

\vspace{2mm}

To analyze the implications, consider first the case
$n+2 \equiv n-2 \equiv 0 \tmod 4)$. We find that, by \eqref{eq:y1^2+y2^2=n-2,n+2},
$y_1,y_2, y_3, y_{4}$ all must be even, and the parity condition is
satisfied automatically.


Otherwise, consider the case $n+2 \equiv n-2 \equiv 1 \tmod 4)$. By using \eqref{eq:y1^2+y2^2=n-2,n+2} again, we find that
$y_{1},y_{2}$ must be of opposite parity, and the same holds for
$y_{3},y_{4}$. The parity conditions $y_{1} \equiv y_{4} \tmod 2)$ and
$y_{2} \equiv y_{3} \tmod 2)$ are now nontrivial, and, by symmetry, only half the
solutions on the two circles yield hyperbolic lattice points. The parity condition can be ensured, if necessary, by
interchanging $y_{1}$ and $y_{2}$.

\subsection{Proof of Proposition \ref{prop:keyprop}}

We next relate the distribution of hyperbolic angles to the
distribution of angles of Euclidean lattice points, with even
$x$-coordinates, on {\em one} circle.

\begin{proof}
We first show $L_n \subseteq \mathcal L_n$.
Let $y_{1}, \ldots, y_{4}$ be as in \eqref{eq:ydef}. We find that
\begin{eqnarray}
(y_{1}+iy_{2})(y_{3}+iy_{4}) &=&
y_{1} y_{3} - y_{2} y_{4} + i( y_{1} y_{4} + y_{2} y_{3}) \nonumber \\
&=&
(a+d)(b+c) - (b-c)(a-d) \nonumber \\
&&+ i( (a+d)(a-d) + (b-c)(b+c)) \label{eq:factor} \\
&=&
2(ac+bd) + i ( a^{2}+b^{2}-c^{2}-d^{2} ) \nonumber \\
&=&
2x_{3} + i x_{4} \nonumber
\end{eqnarray}
with $x_{3},x_{4}$ as in (\ref{eq:x-forms}); it is straightforward to
check that $2x_{3}^{2} + x_{4}^{2} = n^{2}-4$, so $L_n \subseteq \mathcal L_n$.

Let $(2u,v) \in \mathbb Z^2$ with $4u^2+v^2=n^2-4$.
Since $(n-2,n+2)|4$,
we have that $n^2-4 \in \mathcal S$ if and only if $n\pm 2 \in \mathcal S$. It follows that there exist integers $y_1,y_2,y_3,y_4$ with
\[
y_1^2+y_2^2=n-2, \qquad y_3^2+y_4^2=n+2,
\]
and $(y_1+iy_2)(y_3+iy_4)=2u+iv$. In particular, $y_1y_3-y_2y_4$ must be even. Recall from Section \ref{sec:parity} that we only need to
consider $n + 2\equiv n-2 \equiv 0,1 \tmod 4)$. Hence, since $y_1y_3-y_2y_4$ is even, the parity
conditions $y_1 \equiv y_4 \tmod 2)$ and $y_2 \equiv y_3 \tmod 2)$ hold (interchanging $y_1$ and $y_2$ if needed). Thus, for $a,b,c,d$ as in \eqref{eq:abcddef} we have that $a^2+b^2+c^2+d^2=n$, $ad-bc=1$ and $a,b,c,d \in \mathbb Z$. By repeating the calculation performed
in \eqref{eq:factor} we have $2u=2(ac+bd)$ and $v=a^2+b^2-c^2-d^2$. Therefore, $\mathcal L_n \subseteq L_n$.

\end{proof}

\section{Equidistribution discrepancy estimate: Proof of Theorem \ref{thm:equidistributed}}

\subsection{Auxiliary notation}

For $\vec x \in \mathbb R^2$ let $\theta(\vec
x)$ denote the angle between $\vec x$ and the positive $x$-axis. Also,
let $1_S$ denote the indicator function of a set $S$. We also write
\begin{align*}
\omega_1(n)=|\{ p|n : p \equiv 1 \tmod 4)\}|, \quad \text{ and } \quad
\Omega_1(n)=\sum_{\substack{p^a || n \\ p \equiv 1 \tmod 4)}} a,
\end{align*}
where $p^a || n$, means $p^a |n$ and $p^{a+1} \nmid n$. Further, let
\begin{equation}\label{eq:rstardef}
r^{\star}(n)=\sum_{\substack{\vec x=(x,y) \in \mathbb Z^2 \\ |\vec x|^2 =n \\ x \text{ is even}}} 1.
\end{equation}
If $n$ is odd then $r^{\star}(n)=\tfrac12r(n)$ since for every $\vec x=(x,y)$ with $|\vec x|^2=n$ either $x$ or $y$ is even, so by symmetry exactly one-half of the lattice points on the circle of radius $\sqrt{n}$ will have even $x$-coordinate.  Next, consider $n=2m$ with $m$ odd then $r^{\star}(n)=0$ since for $|\vec x|^2=2m$ there exists $a,b\in \mathbb Z$ with $x+iy=(1+i)(a+ib)=(a-b)+i(b+a)$ and either $a$ or $b$ is even ($m$ is odd) so $x$ and $y$ are both odd. Finally, if $n=4m$ and $x^2+y^2=4m$, then $x+iy=2(a+ib)$ for some $a,b \in \mathbb Z$, so $x,y$ are even and in this case $r^{\star}(n)=r(n)$. In summary,
\begin{equation} \label{eq:rstar}
r^{\star}(n)=
\begin{cases}
\tfrac12 \cdot r(n) & \text{ if } n \equiv \pm 1 \pamod 4, \\
0 & \text{ if } n \equiv 2 \pamod 4,\\
r(n) & \text{ if } n \equiv 0 \pamod 4.
\end{cases}
\end{equation}
 Additionally, let
\begin{align*}
\tau_{\ell}(n)=\sum_{n_1 \cdots n_{\ell}=n} 1
\end{align*}
denote the $\ell$-fold divisor function, and define $b(n)=1$ if $n$ is a sum of two squares and to be equal to zero otherwise (i.e. $b(\cdot)$ is the characteristic function of $\mathcal{S} = \{ n = \square +  \square\}$).

\begin{definition}[Primary numbers]
A Gaussian integer $a + ib \in \mathbb Z[i]$ is \textit{primary} if $a + ib \equiv 1 \tmod{ 2(1 + i)})$.
Equivalently, $a+ib$ is primary if
\[
\begin{split}
    b \equiv& 0 \tmod 2) \\
    a\equiv& 1-b \tmod 4).
\end{split}
\]
\end{definition}

It is well-known that a product of two primary Gaussian integers is primary,
and each primary Gaussian integer can be factored uniquely, up to reordering, into primary Gaussian
primes, see e.g. \cite[p. 54]{iwaniec-kowalski}. For a
prime $p \equiv 1 \tmod 4)$, let $\pi$ be the
primary Gaussian prime with norm $N\pi=p$ and $\Im (\pi)>0$. We define
\begin{equation} \label{eq:thetadef}
\theta_{p}=\tmop{Arg}(\pi)
\end{equation}
where $\tmop{Arg}(\cdot)$ is the principal value of the argument function.

\subsection{Proof of Theorem \ref{thm:equidistributed}}

By Proposition \ref{prop:keyprop},
Theorem~\ref{thm:equidistributed} is an immediate consequence of the
following discrepancy bound for integral lattice points, having
even $x$-coordinate, on circles with radii $\sqrt{n^2-4}$.
\begin{prop} \label{prop:1}
Along a density one subsequence of integers $n$ such that $n^2-4$ is a sum of two squares we have that $r^{\star}(n^2-4) \asymp (\log n)^{\log 2 \pm o(1)}$ and moreover
\begin{equation} \label{eq:discrepancy}
\sup\limits_{I \subseteq S^1} \bigg| \frac{1}{r^{\star}(n^2-4)} \sum_{\substack{\vec x=(x,y) \in \mathbb Z^2 \\ |\vec x|^2 =n^2-4 \\ x \text{ is even}}} 1_{I}(\theta(\vec x))-\frac{|I|}{2\pi} \bigg| \ll  \frac{1}{r^{\star}(n^2-4)^{\vartheta-o(1)}},
\end{equation}
where $\vartheta=\log(\pi/2)/\log 2=0.651496129\ldots.$
\end{prop}

Towards a proof of Proposition \ref{prop:1} we record the following consequence of the work of Nair-Tenenbaum
\cite[Eq. (7)]{nair-tenenbaum}.
Let $h \in \mathbb Z$ and $f,g$ be non-negative multiplicative functions such that $f(n) \le \tau_{\ell}(n)$ and $g(n) \le \tau_{\ell}(n)$ for some $\ell$. Then
\begin{equation} \label{eq:NT}
\sum_{n \le X} f(n)g(n+h) \ll X \prod_{p \le X} \left(1+\frac{f(p)-1}{p} \right)\left(1+ \frac{g(p)-1}{p} \right)
\end{equation}
where the implied constant depends on $h$ (alternatively see \cite[Theorem 15.6]{cribro}, which would suffice for our purposes). Additionally, let us recall Hooley's result \cite[Theorem 3]{hooley} that
\begin{equation} \label{eq:Hooley}
\sum_{n \le X} b(n-2)b(n+2) \gg \frac{X}{\log X}.
\end{equation}
Finally we record the following estimate, which follows from the work of
Erd\H{o}s-Hall \cite{erdos-hall}. Given an odd prime $p \equiv 1 \tmod
4)$, define
\begin{equation} \label{eq:varthetadef}
\vartheta_p:=\arctan(y/x)
\end{equation}
where $p =x^2+y^2$ with $0 \le y \le x$, and note that
$|\cos(k\theta_p)|=|\cos(k\vartheta_p)|$ for $k \in \mathbb Z$ with $2|k$.  Hence, for any
$k \in \mathbb Z$ with $2|k$ it follows from repeating the argument in
\cite[Eqs. (24)-(25)]{erdos-hall} that
\begin{equation} \label{eq:eh}
\sum_{\substack{p \le X \\ p \equiv 1 \pamod 4}} |\cos(k\theta_p)| =\frac{1}{\pi} \int_2^X \frac{dt}{\log t}+O\bigg(|k| X e^{-c \sqrt{\log X}}\bigg),
\end{equation}
for some $c>0$.

Before proving Proposition \ref{prop:1} we need the following simple estimate.
\begin{lem} \label{lem:normal} Let $\varepsilon>0$ be fixed but sufficiently small. Then
\begin{equation} \label{eq:normal1}
\sum_{\substack{n \le X \\ \omega_1(n^2-4) \le (1-\varepsilon) \log \log X}} b(n^2-4)  \ll \frac{X}{(\log X)^{1+\frac12 \varepsilon^2}}.
\end{equation}
Also
\begin{equation} \label{eq:normal2}
\sum_{\substack{n \le X \\ \Omega_1(n^2-4) \ge (1+\varepsilon) \log \log X}} b(n^2-4)  \ll \frac{X}{(\log X)^{1+\frac13 \varepsilon^2}}.
\end{equation}
\end{lem}
\begin{rem}
For $r^{\star}(n) \neq 0$, we have that
\[
2^{\omega_1(n)} \le r^{\star}(n) \le 2^{\Omega_1(n)+2}.
\]
Hence,
by \eqref{eq:normal1} and \eqref{eq:normal2} we have for all integers $n \le x$ with $r^{\star}(n^2-4) \neq 0$ that outside of an exceptional set of size at most $O(X/(\log X)^{1+\frac13 \varepsilon^2})$
\begin{equation} \label{eq:normalorder}
r^{\star}(n^2-4)\asymp (\log n)^{\log 2 \pm 2\varepsilon}.
\end{equation}
\end{rem}
\begin{proof}
Applying Chernoff's bound and \eqref{eq:NT} we have for any $\alpha>0$
\[
\begin{split}
\sum_{\substack{n \le X \\ \omega_1(n^2-4) \le  (1-\varepsilon) \log \log X}} b(n^2-4) & \le  (\log X)^{(1-\varepsilon)\alpha} \sum_{\substack{n \le X}} b(n^2-4)  e^{-\alpha \omega_1(n^2-4)} \\
&\ll  X (\log X)^{(1-\varepsilon) \alpha} \prod_{p \le X} \left(1 + \frac{b(p)e^{-\alpha}-1}{p} \right)^2 \\
& \ll \frac{X}{\log X} (\log X)^{\alpha(1-\varepsilon)+e^{-\alpha}-1}.
\end{split}
\]
Taking $\alpha=\varepsilon$ it follows that
$
\alpha(1-\varepsilon)+e^{-\alpha}-1 \le -\varepsilon^2/2,
$
which completes the proof of \eqref{eq:normal1}. The proof of \eqref{eq:normal2} follows from a similar argument, which we will omit.
\end{proof}

\begin{proof}[Proof of Proposition \ref{prop:1}]
For $n$ such that $b(n)=1$, $n \not\equiv 2 \tmod 4)$ and $k \in \mathbb Z$ let
\begin{equation} \label{eq:ukdef}
\begin{split}
u_k(n)& :=\frac{1}{r(n)}\sum_{\substack{\vec x\in \mathbb Z^2 \\ |\vec
    x|^2=n}} e^{ik \theta(\vec x)},
\qquad \qquad
v_k(n) :=
\frac{1}{r^{\star}(n)}\sum_{\substack{\vec x=(x,y)\in \mathbb Z^2 \\
    |\vec x|^2=n \\ x \text{ is even}}} e^{ik \theta(\vec x)},
    \end{split} \end{equation}
and
\begin{equation}
\label{eq:wk def}
w_k(n) :=\frac{4}{r(n)}\sum_{\substack{\vec x=(x,y)\in \mathbb
    Z^2 \\ |\vec x|^2=n \\ x+iy \text{ is primary} }} e^{ik
  \theta(\vec x)}.
\end{equation}
The function $w_k(\cdot)$ is multiplicative, and for $4|k$ we have that $u_k(n)$ is multiplicative (if $4\nmid k$ then $u_k(n)=0$).
Also, for any integer $m$, if $k$ is odd then $
v_k(m)=0$, which can be seen by noting that $\theta(-x,-y) \equiv \pi+\theta(x,y) \tmod 2\pi)$; so if $k$ is odd the  terms corresponding to $(x,y)$,$(-x,-y)$ in the sum cancel with one another. Let us also record the following basic property
\begin{equation} \label{eq:real}
    v_k(n)=v_{-k}(n),
\end{equation}
which follows from making the change of variables $(x,y)\rightarrow (x,-y)$.

\vspace{2mm}

Since $x+iy$ is primary if and only if $x-iy$ is primary, by grouping together terms with their conjugates it follows that $w_k(n)$ is real-valued, so that
\begin{equation} \label{eq:real2}
w_k(n)=\overline{ w_{k}(n)}=w_{-k}(n).
\end{equation}
Below we will prove that if $n$ is odd, then
\begin{equation} \label{eq:claim1}
v_k(n)=(-1)^{k/2}  w_k(n),
\end{equation}
for $2|k$.
Hence, by the Erd\H{o}s-Turan inequality (see \cite[Corollary 1.1]{montgomery}) and \eqref{eq:claim1}, it follows that the l.h.s. in \eqref{eq:discrepancy} is, for odd $n \le X$, bounded by
\begin{align*}
\ll \frac{1}{\log X}+\sum_{\substack{1 \le k \le \log X \\ 2|k}} \frac{| w_{k}(n^2-4)|}{k}.
\end{align*}
Let $\varepsilon>0$ be sufficiently small but fixed and let
\begin{align*}
N(X)=\left\{n \le X:  b(n^2-4)=1 \text{ and } \omega_1(n^2-4) \ge (1-\varepsilon)\log \log X\right\},
\end{align*}
\begin{align*}
N_{\text{odd}}(X)=\left\{ n \in  N(X) : n \text{ is odd}\right\},
\end{align*}
and $N_{\text{even}}(X)$ be defined similarly.
By Lemma \ref{lem:normal} and \eqref{eq:Hooley},
\begin{equation} \label{eq:zerodensity}
\frac{1}{| \{ n \le X : b(n^2-4)=1\}|} \sum_{\substack{n \le X \\ \omega_1(n^2-4) \le (1-\varepsilon) \log \log X}} b(n^2-4) \ll \frac{1}{(\log X)^{\frac12 \varepsilon^2}},
\end{equation}
so, for results concerning a density one subsequence of integers with $b(n^2-4)=1$, it suffices
to consider $n \in N(X)$.

Applying Chebyshev's inequality we get that
\begin{equation} \label{eq:cheby}
\begin{split}
& \bigg| \bigg\{ n \in N_{\text{odd}}(X) : \sum_{\substack{1 \le k \le \log X \\ 2|k}} \frac{|w_k(n^2-4)|}{k}  \ge (\log X)^{-\log \frac{\pi}{2}+\varepsilon} \bigg\}\bigg| \\
&\le  (\log X)^{\log \frac{\pi}{2}-\varepsilon} \sum_{\substack{1 \le k \le \log X \\ 2|k}} \frac{1}{k}
\sum_{\substack{n \le X \\ \omega_1(n^2-4) \ge (1-\varepsilon) \log \log X \\ n \text{ is odd} }} \left|w_k(n^2-4)\right|.
\end{split}
\end{equation}
To bound the sum on the r.h.s. of \eqref{eq:cheby} we apply Chernoff's bound and \eqref{eq:NT} to get for any $\alpha >0$ and $1 \le k \le \log X$ that
\[
\begin{split}
\sum_{\substack{n \le X \\ \omega_1(n^2-4) \ge (1-\varepsilon) \log \log X \\ n \text{ is odd} }} |w_k(n^2-4)|
\le & \frac{1}{(\log X)^{\alpha(1-\varepsilon)}}\sum_{\substack{n \le X \\  n \text{ is odd}}} |w_k(n+2)|e^{\alpha \omega_1(n+2)} \cdot |w_k(n-2)| e^{\alpha \omega_1(n-2)} \\
\ll & \frac{X}{(\log X)^{\alpha(1-\varepsilon)}} \prod_{p \le X} \left(1 + \frac{|w_k(p)| e^{\alpha} b(p)-1}{p}\right)^2.
\end{split}
\]
Applying \eqref{eq:eh}, the above quantity is bounded by
\[
\ll \frac{X}{\log X} (\log X)^{e^{\alpha} \frac{2}{\pi}-\alpha-1+\alpha \varepsilon},
\]
and after taking $\alpha=\log \pi/2$ we get that this is $\ll \frac{X}{\log X} (\log X)^{(-1+\varepsilon)\log \pi/2}.$
 Using this estimate in \eqref{eq:cheby} along with \eqref{eq:normalorder} and \eqref{eq:zerodensity} shows that the bound \eqref{eq:discrepancy} holds for all odd $n \le X$ such that $b(n^2-4)=1$ outside a set of size
\[
\ll \frac{X (\log \log X)}{(\log X)^{1+\varepsilon-\varepsilon \log \pi/2}}+\frac{X}{(\log X)^{1+\frac13\varepsilon^2}}.
\]
This completes the proof of Proposition \ref{prop:1} for odd $n$.

We now show that for almost all even $n$ the discrepancy (i.e. \eqref{eq:discrepancy}) is small.
For even $n=2m$ we have $4|n^2-4$, hence
$v_k(n^2-4)=u_k(n^2-4)$.
Also,
\[
|u_k(n^2-4)| \le |u_k(m+1)\cdot u_k(m-1)|,
\]
which follows from the multiplicativity of $u_k(n)$ and the observation that
$|u_k(2^{a})|=1$ for $4|k$,
which can be seen by analyzing the solutions $x,y \in \mathbb Z$ to $x^2+y^2=2^a$ (recall if $4\nmid k$ then $u_k(n)=0$). Hence, proceeding as before, we obtain the desired discrepancy bound for almost all even $n$ as well.

It remains to establish the claim, i.e. \eqref{eq:claim1}. To see this first note that for a function with $f(z)=f(\overline z)$ we have
\begin{equation} \label{eq:observation}
\begin{split}
 \sum_{\substack{a+ib \in \mathbb Z[i] \\ a \text{ is even} \\ b \text{ is odd}}} f(a+ib)=&   \sum_{\substack{a+ib \in \mathbb Z[i] \\ a \text{ is odd} \\ b \text{ is even}}} f(i(a+ib)) \\
=&  \sum_{\substack{a+ib \in \mathbb Z[i] \\ a \equiv 1-b \pamod 4 \\ b \text{ is even}}} f(-b+ai)+   \sum_{\substack{a+ib \in \mathbb Z[i] \\ -a \equiv 1-b \pamod 4  \\ b \text{ is even}}} f(-b+ai) \\
   =& \sum_{\substack{a+ib \in \mathbb Z[i] \\ a \equiv 1-b \pamod 4 \\ b \text{ is even}}} f(-b+ai)+   \sum_{\substack{a+ib \in \mathbb Z[i] \\ a \equiv 1-b \pamod 4  \\ b \text{ is even}}} f(-b-ai) \\
   =& 2 \sum_{\substack{a+ib \in \mathbb Z[i] \\ a+ib \text{ is primary}}} f(i(a+ib)),
   \end{split}
\end{equation}
provided the sums above are absolutely convergent. Taking
\begin{align*}
f(z)=(e^{i k \arg(z)}+e^{-i k \tmop{arg}(z)})1_{|z|^2=n},
\end{align*}
with $2|k$, and applying \eqref{eq:rstar}, \eqref{eq:real}, \eqref{eq:real2} and \eqref{eq:observation} we get that
\[
2v_k(n)=v_{k}(n)+v_{-k}(n)=2(-1)^{k/2}(w_k(n)+w_{-k}(n)) \frac{r(n)}{4 r^{\star}(n)}=2(-1)^{k/2}  w_k(n),
\]
which establishes \eqref{eq:claim1}.
\end{proof}

\section{Non-uniform limits: proof of Theorem \ref{thm:non-equidist}}

In this subsection we show that there exist sparse subsequences of integers $\{n_j\}_j$ such that $r^{\star}(n_j^2-4)\rightarrow \infty$ as $j \rightarrow \infty$, and the integral lattice points on circles of radii $\sqrt{n_j^2-4}$ with even $x$-coordinates fail to equidistribute as $j \rightarrow \infty$.
A key ingredient (Lemma \ref{lem:klr} below) is a result
of Kurlberg-Lester-Rosenzweig \cite{KLR}, which builds
on related works of Huxley-Iwaniec \cite{huxley-iwaniec} and
Friedlander-Iwaniec \cite[Theorem 14.8]{cribro}.



We also use a construction, which exploits the fact that Gaussian primes are equidistributed in narrow sectors.
This follows from work of Kubilius \cite{kubilius} who proved that
\[
\left| \left\{ p \le X : p \equiv 1 \tmod{4}) \text{ and } |\vartheta_p| \le X^{-1/10} \right\} \right|\gg \frac{X^{9/10}}{\log X}.
\]
Using the estimate above, it follows that there exists $Q < (\log X)^{1/{10}}$ which is squarefree such that $|\{p: p|Q\}| \asymp \log \log \log X$ and if $p|Q$ then $\log \log X \le p \le (\log \log X)^2$, $p \equiv 1 \tmod 4)$
and $|\vartheta_p| \le (\log \log X)^{-1/10}$.

\begin{lem} \label{lem:klr}
Let  $m_0$ be an integer such that $p|m_0$ implies $p \equiv 1 \tmod 4)$. Suppose $m_0=f^2 e$ where $e$ is squarefree, $1 \le e \le \sqrt{\log \log X}$ and $f \ll 1$.
Also, let $Q$ be as above.
Then there exist  $\gg \frac{X}{(\log X)^{2+o(1)}}$ integers $n \le X$ such that
\[
n^2-4=m_0 p_1p_2 Q \ell_n
\]
where $p_1,p_2$ are distinct primes $\equiv 1 \tmod 4)$ with $|\vartheta_p| \le (\log \log X)^{-1/2}$, $p_1,p_2 \ge \log X$ and $r^{\star}(\ell_n)\asymp 1$. Additionally, $m_0, \ell_n, p_1,p_2, Q$ are pairwise co-prime.
\end{lem}
\begin{proof}
This follows from \cite[Proposition 2.1]{KLR} with $Q_0=1$ and $Q_1=m_0Q$ and $\varepsilon=(\log\log X)^{-1/2}$, and here $n-2=p_1p_2$ and $n+2=m_0Q \ell_n$, so $(p_1p_2, \ell_n)=1$. Note that each of the prime factors of $\ell_n$ are $\ge X^{\eta}$ for some small but fixed $\eta>0$ so that $(\ell_n,Q)=(\ell_n,m_0)=(Q,m_0)=1$.
\end{proof}

\begin{proof}[Proof of Theorem \ref{thm:non-equidist}]
By Lemma \ref{lem:klr} there exist $\gg X/(\log X)^{2+o(1)}$ integers $n \le X$ such that $n^2-4=m_0 \ell_n  Q'$, where $Q'=Q p_1p_2$.
 Moreover $Q',\ell_n,m_0$ are pairwise co-prime. Also, for a prime $p|Q'$ we have $\theta_p=l_p \frac{\pi}{2}+O((\log \log X)^{-1/10})$ for some $l_p \in \mathbb Z$, where $\theta_p$ is as defined in \eqref{eq:thetadef}. Let $u_n=(-1)^{\sum_{p|Q'} \ell_p}$ and using the fact that $w_k(\cdot)$ is multiplicative we have for $2|k$
 \[
 w_k(Q')= \prod_{p|Q'} \cos(k \theta_p)=u_n^{k/2}+o(1).
 \]
 Hence, for $n$ as above we have for $2|k$ that
 \[
 w_k(n^2-4)=w_k(Q') w_k(\ell_n m_0)=(u_n^{k/2}+o(1))w_k(\ell_nm_0).
 \]
 Using this \eqref{eq:rstar} and \eqref{eq:claim1} it follows for $n$ as above and $2|k$ that
\begin{equation} \label{eq:fourier}
\begin{split}
 v_k(n^2-4)
 =&(-1)^{k/2}  w_k(n^2-4) \\
 =&  u_n^{k/2} (-1)^{k/2}  w_k(\ell_n m_0)+o(1)
 \end{split}
 \end{equation}
Since the angles of primary primes equidistribute in sectors\footnote{To see this, note that $\chi_k((\alpha))= \left(\frac{\alpha}{|\alpha|}\right)^k$, where $\alpha$ is the primary generator of $(\alpha)$, is a Hecke grossencharacter $\tmod \mathfrak m)$ of frequency $k$ for any $k \in \mathbb Z$ where $\mathfrak m=(1)$ if $4|k$ and $\mathfrak m=2(1+i)$ if $4 \nmid k$ \cite[Eq'n (3.91)]{iwaniec-kowalski}. Hence, equidistribution of angles of primary primes follows from the standard zero free region for the $L$-functions attached to these characters.}
we know there exist distinct primes $q_1,q_2$ such that $\theta_{q_j}=\tfrac{\pi}{4}+o(1)$, as $q_j \rightarrow \infty$, for $j=1,2$. Take $m_0=q_1q_2 \widetilde m_0$ with $(\widetilde m_0,q_1q_2)=1$. Also, $2|k$ so $e^{ik \theta}=e^{ik(\theta+\pi)}$ for any $\theta \in \mathbb R$. Hence,
\[
\begin{split}
w_k(q_1q_2 \widetilde m_0)=&w_k(q_1q_2)w_k(\widetilde m_0)= \frac{1}{4}(e^{ik \frac{\pi}{4}}+e^{-ik \frac{\pi}{4}})^2  w_k(\widetilde m_0)+o(1) \\
=& \frac{1}{r(\widetilde m_0)} \sum_{\substack{\vec{x}=(x,y) \in \mathbb Z \\ |\vec x|^2=n \\ x+iy \text{ is primary }}}\left(e^{ik \theta(\vec x)}+e^{ik( \theta(\vec x)+\pi)}+e^{ik( \theta(\vec x)+\frac{\pi}{2})}+e^{ik (\theta(\vec x)-\frac{\pi}{2})}\right)+o(1) \\
=& u_k(\widetilde m_0)+o(1),
\end{split}
\]
where $u_k$ is as given in \eqref{eq:ukdef} (note that the definition extends to all $n$ with $b(n)=1$).
Therefore using this along with \eqref{eq:claim1} and \eqref{eq:fourier} we conclude for $2|k$ that
\begin{equation} \label{eq:fourierconcluded}
\begin{split}
v_k(n^2-4)=& u_k(\widetilde m_0) u_n^{k/2} (-1)^{k/2} w_k(\ell_n)+o(1).
\end{split}
\end{equation}

Let $\nu$ be a probability measure on $S^1$, which is attainable. Since $\nu$ is invariant under rotation by $\tfrac{\pi}{2}$ we have $\widehat \nu(k)=0$ if $4 \nmid k$, where $\widehat \nu(k)= \frac{1}{2\pi}\int_{S^1} e^{-ikt} d\nu(t)$. There exists a sequence of integers $\{m_j\}$
such that for each $k \in \mathbb Z$
\begin{equation}\label{eq:measure}
\lim_{j \rightarrow \infty} u_k(m_j)=\widehat \nu(-k).
\end{equation}
Write $m_j=2^{a_j} l_j \widetilde m_j $ where $p| l_j$ implies $p \equiv 3 \tmod 4)$ and $p|\widetilde m_j $ implies $p \equiv 1 \tmod 4)$.
Observe that if $p \equiv 3 \tmod 4)$ then $u_k(p^a)=(1+(-1)^a)/2$ so $u_k(m_j)=u_k(2^{a_j} \widetilde m_j)$. Hence, taking $\widetilde m_0=\widetilde m_0(j)= q_0^{a_j} \widetilde m_j$ where $q_0 \equiv 1 \tmod 4)$ , with $\theta_{q_0}=\tfrac{\pi}{4}+o(1)$ and $(q_0, q_1q_2 \widetilde m_j)=1$ we have that
\begin{equation} \label{eq:attained}
u_k(m_j)=u_k(\widetilde m_0)+o(1).
\end{equation}
Also, by passing to a subsequence $\{n_j\}$ of the integers $n$ as above there exists a probability measure $\widetilde \nu$ on $S^1$, which is supported on at most $O(1)$ points, such that for each $k \in \mathbb Z$
\begin{equation} \label{eq:finite}
\lim_{j \rightarrow \infty} u_{n_j}^{k/2} (-1)^{k/2} w_k(\ell_{n_j})  = \widehat{ \widetilde \nu}(-k).
\end{equation}
Hence, by \eqref{eq:fourierconcluded}, \eqref{eq:measure}, \eqref{eq:attained} and \eqref{eq:finite} we have for each $k \in \mathbb Z$ that
\[
\lim_{j \rightarrow \infty} \nu_k(n_j^2-4)=\widehat \nu(-k) \widehat{\widetilde \nu}(-k).
\]
Since by Proposition \ref{prop:keyprop}, $\widehat \mu_{n_j}(-k)=v_{k}(n_j^2-4)$ we conclude that $\mu_{n_j} \Rightarrow \nu \ast \widetilde \nu$, as claimed. \end{proof}

\section{Beyond symmetry: proof of Theorem \ref{thm:break}}
\label{sec:break-rotat-symm}

\subsection{Outline of the proof of Theorem \ref{thm:break}}

As we have seen, by Proposition \ref{prop:keyprop}, the measures $\mu_{n}$ can be described in terms of
the angles from lattice points on Euclidean circles, namely the set of
points
\begin{equation} \label{eq:lattice}
\left\{ (x,y)/\sqrt{n^{2}-4} : (x,y) \in \Z^2 : x^{2}+y^{2} = n^{2}-4, x
\equiv 0 \tmod 2)\right\}.
\end{equation}
We will work with this interpretation.
The proof of Theorem \ref{thm:break} uses a construction based on producing a sequence of odd integers $n$
such that $r(n^{2}-4) >0$ and $\Omega(n^{2}-4)$ is bounded, and then
showing that almost all of the corresponding measures $\mu_{n}$ are
asymmetric. Intuitively this should not be surprising --- measures
supported on a finite number of points ought to be very unlikely to be
symmetric. However, we observe that this construction can be
modified to show the existence of asymmetric measures while allowing
$r(n^{2}-4)$ to grow, by ensuring $Q | n-2$ where $Q$ is a product of
a growing number of primes $p$ with $\vartheta_{p}$ very small (the
details will be left to the interested reader).
%
%
For the existence part we use a lower bound sieve.  To show that most
measures within this sequence break symmetry, we use an upper bound
sieve to show that for rather few such $n$'s, $n^{2}-4$ is divisible
by a ``thin'' set of primes, of relative density $\delta$ for
$\delta>0$ small.

\subsection{Preliminaries}
Since multiplication of a Gaussian integer by $i$ interchanges the parities of the real and
imaginary parts, we can just as well study the distribution of angles
of points having even imaginary part.
The advantage here is that this set is closed under multiplication.
Define a function
$\chi_{2}(z) := (z/|z|)^{2}$ on $\Z[i] \setminus \{ 0\}$, and let
$$
W_{2}(n) := \frac{1}{2}
\sum_{\substack{z \in \Z[i] : |z|^{2}=n \\ \Im(z) \equiv 0 \tmod 2)}}
\chi_2(z).
$$
By using the relation with Euclidean lattice points (see \eqref{eq:lattice}) it is not hard to see that the measure $\mu_{n}$ is asymmetric if $W_2(n^2-4) \neq 0$, and we will justify this later.

To analyze $W_2(n)$, first note that
the set $\{z \in \mathbb Z[i]:\Im(z) \equiv 0 \tmod 2)\}$ is closed under
multiplication, and for any element $z$ of this set we can write
$$
z =  u \prod_{\pi_{i}|z}  \pi_{i}^{e_{i}}
$$
with $\pi_{i}$ ranging over Gaussian primes with $\Im(\pi_i)
\equiv 0 \tmod 2)$
and $u=u(z) \in \{-1,1\}$.  In particular, $\chi_2(u)=1$, which for
our purposes resolves the indeterminacy of the sign.
If $(p) = (\pi) \overline{(\pi)}$ (as ideals) for $p$ split, it will
be convenient to use the following sign convention: the sign of the representative $\pi$ of the ideal $(\pi)$ with $\Im(\pi)$ even can be picked
arbitrarily, and then we choose the representative of the
ideal $(\overline{\pi})$ to be $\overline{\pi}$ (note that complex
conjugation preserves the parity of the imaginary part).
Further, $W_2(n)$ is multiplicative. We find that
$$
W_{2}(p) =
(\pi^{2} + (-\pi)^{2} +  \overline{\pi^{2}} + (-
\overline{\pi})^{2})/(2p)
=
(\pi^{2} + \overline{\pi^{2}})/p
= 2\cos( 2 \theta_{p}),
$$
where $\theta_p$ is as in \eqref{eq:thetadef}.
Moreover, it is not hard to see that for any integer $j \ge 0$
\begin{equation} \label{eq:chebypol}
W_2(p^j)= \frac{1}{p^j}\sum_{\ell=0}^{j} \pi^{2\ell} \overline{\pi}^{2j-2\ell}= \frac{\sin((j+1)2\theta_p)}{\sin(2\theta_p)},
\end{equation}
where the last step follows from evaluating the geometric sum (note that $p^j=\pi^j \overline \pi^j).$
Also if $p \equiv 3 \tmod 4)$ then
\begin{equation} \label{eq:inert}
W_2(p^j)=(1+(-1)^j)/2.
\end{equation}
Hence, in
order to bound $W_{2}(n^{2}-4)$ away from zero, it is enough to bound
$W_{2}(p^j)$ away from zero for all $p^j || n^{2}-4$ since we
consider odd $n$ for which $\Omega(n^{2}-4) =O(1)$.

As previously mentioned, our construction uses estimates provided by upper and lower bound sieves. To state these results we will first introduce some notation.
Let $B_{0}$ be a sufficiently large integer, and let $\eta_1,\eta_2>0$ be
sufficiently small numbers with $\eta_1< \eta_2$. Define, for $x$ large and $\varepsilon>0$ fixed,
\begin{eqnarray}
P_{\varepsilon}:= \left\{ p \in [ (\log x)^{B_{0}}, x] : p \equiv 1 \tmod
4), |\vartheta_p| \le \varepsilon\right\},
\end{eqnarray}
where $\vartheta_p$ is as in \eqref{eq:varthetadef}. Let
\[
P_{\varepsilon}' :=
\left\{ p  \in P_{\varepsilon} : p \le x^{1/9}\right\},
\quad
P_{\varepsilon}''
: =
\left\{ p \in P_{\varepsilon} : (\log x)^{B_0} \le p \le x^{\eta_1} \right\}
\]
and for $z>0$, put
\[
P(z):=\prod_{p \le z} p.
\]
Further, for $j \in \mathbb N$ let
\[
Q_{\delta,j}=\left\{q \equiv 1 \tmod 4): \text{ $q$ prime,}
\min_{k \in \mathbb Z} |\theta_q-\tfrac{\pi k}{2(j+1)}|<\frac{\delta}{2}\right\}.
\]
Also, define
\[
M = M(x) :=
\left\{ m \leq x : m = p_{1} p_{2}, p_{1} \in P_{\varepsilon},
p_{2} \in P_{\varepsilon}'', b(m+4) = 1, (m+4,P(x^{\eta_2}))=1 \right\}
\]
and  given $\delta>0$, let
$$
M_{\delta,j} = \{ m \in M : q | m+4 \text{ for some $q\in Q_{\delta,j}$} \}.
$$
\begin{prop}
\label{prop:m-size}
For $\varepsilon>0$ fixed, we have
$$
|M| \gg \varepsilon^{2} \frac{x (\log \log x)}{(\log x)^{2}},
$$
and, for $j \in \mathbb N$ and fixed $\delta>0$,
$$
|M_{\delta,j}| \ll
j \delta \frac{x (\log \log x)}{(\log x)^{2}}
$$
where the implied constant depends at most on $\eta_1$ and $\eta_2$.

\end{prop}

We will now deduce Theorem \ref{thm:break} from Proposition \ref{prop:m-size}.

\subsection{Proof of Theorem \ref{thm:break}}

\begin{proof}
We begin by noting that for a symmetric measure $\nu$, $d\nu$ is
invariant under the change of variables $t \rightarrow t+\pi/2$.
Hence, $ \frac{1}{2\pi}\int_{S^1} e^{2it} d\nu(t)=- \frac{1}{2\pi}\int_{S^1} e^{2it} d\nu(t)$ and
$\widehat \nu(-2)=0$. Thus, a measure $\nu$ is asymmetric provided
$\widehat \nu(-2) \neq 0$.

Let $M_{\delta}=\cup_{0 \le j \le 1/\eta_2} M_{\delta,j}$ and set $\delta=\varepsilon^3$. Then by Proposition \ref{prop:m-size}, for $\varepsilon$ sufficiently small, we have
\[
|M \setminus M_{\delta}| \gg \varepsilon^2 \frac{x(\log \log x)}{(\log x)^2}.
\]
 Also, by construction $\Omega(m+4) \le 1/\eta_2$. Hence, for $m \in M\setminus M_{\delta}$ it follows from \eqref{eq:chebypol} and \eqref{eq:inert} that
\[
|W_2(m(m+4))|= |W_2(m) \prod_{p^j||m+4} W_2(p^j)| \gg \delta^{1/\eta_2}.
\]
Thus, for $n=m+2$ the measure $\mu_{n}$ is asymmetric, for each $m \in M\setminus M_{\delta}$ since $W_2(n^2-4)$ is uniformly (in $n$) bounded away from $0$. It follows that there exists a subsequence $\{n_i\}$ of such integers and a probability measure $\mu$ on $S^1$ such that $\mu_{n_i}$ weakly converges to $\mu$. Let $r^{\star}(n)$ be as in \eqref{eq:rstardef}. Moreover, using \eqref{eq:lattice} we conclude that
\[
\begin{split}
\widehat \mu(-2)=\frac{1}{2\pi}\int_{S^1} e^{2it} \, d\mu(t)&=
\lim_{i \rightarrow \infty} \frac{1}{r^{\star}(n_i^2-4)} \sum_{\substack{|z|^2=n_i^2-4 \\ \Re(z)\equiv 0 \pamod 2}} \chi_2(z) \\
&=
\lim_{i \rightarrow \infty} \frac{-2W_2(n_i^2-4)}{r^{\star}(n_i^2-4)} \neq 0,
\end{split}
\]
where in the last step we made the change of variables $z
\rightarrow iz$ and used that $\chi_2(iz)=-\chi_2(z)$.
Therefore, $\mu$ is asymmetric. \end{proof}


\subsection{Proof of Proposition~\ref{prop:m-size}}
\label{sec:proof-proposition}
Given $P \subset \mathbb N$ let $1_P$ denote the indicator function of $P$. Also, for two arithmetic functions $f,g$ let $f\ast g$ denote the Dirichlet convolution of $f$ with $g$.

\begin{lem}\label{lem:firstlemma}
For $\varepsilon>0$ fixed we have
\[
\sum_{\substack{n \le x \\ (n+4,P(x^{\eta_2}))=1}}
(1_{P_{\varepsilon}}\ast 1_{P_{\varepsilon}''})(n) b(n+4)
\gg
\varepsilon^2 \frac{x (\log \log x)}{(\log x)^2}.
\]

\end{lem}

\begin{proof}
  By \cite[Proposition~2.1]{KLR} with $Q_0=Q_1=1$, we have
$$
\sum_{\substack{n \le x \\ (n+4,P(x^{\eta_2}))=1}}
(1_{P_{\varepsilon}}\ast 1_{P_{\varepsilon}'})(n) b(n+4)
\gg
\varepsilon^2 \frac{x (\log \log x)}{(\log x)^2},
$$
where the implied constant depends at most on $\eta_2$.
The contribution from the terms with $n=p_1p_2$ and $x^{\eta_1} \le p_2 \le x^{1/9}$ is
then $O(x/(\log x)^2)$.  To see this, note that the
upper bound sieve (cf. \cite[Theorem~6.9]{cribro}) gives
$$\displaystyle \sum_{\substack{n \le x/p \\ (n(pn+4),
    P(x^{\eta_2}))=1}} 1 \ll \frac{x}{p (\log x)^2}$$
provided $\eta_2$ is sufficiently small and $p \le x^{1/9}$.  The result
follows by summing over
$x^{\eta_1} \le p \le x^{1/9}$.
\end{proof}

\begin{lem} \label{lem:secondlemma}
Let $q \in [x^{\eta_2}, 2x^{1-\eta_2}]$ be a prime. Then for $\eta_1$ sufficiently small in terms of $\eta_2$
we have that
\[
\sum_{\substack{n \le x \\ (n+4,P(x^{\eta_2}))=1 \\ q|n+4}} (1_{P_{\varepsilon}}\ast 1_{P_{\varepsilon}''})(n) b(n+4)  \ll  \frac{x (\log \log x)}{q(\log x)^2}
\]
where the implied constant depends at most on $\eta_1$ and $\eta_2$.
\end{lem}
\begin{proof}
Let $\eta_1 < \theta <\eta_2$ be chosen later. Take $z=x^{\theta}$
and note that if $(n+4, P(x^{\eta_2}))=1$ then $(n+4, P(z))=1$,
hence
\[
\sum_{\substack{n \le x \\ (n+4,P(x^{\eta_2}))=1 \\ q|n+4}}
(1_{P_{\varepsilon}}\ast 1_{P_{\varepsilon}''})(n) b(n+4) \le
\sum_{\substack{n \le x \\ (n+4,P(z))=1 \\ q|n+4}}
(1_{P_{\varepsilon}}\ast 1_{P_{\varepsilon}''})(n)
\]
\begin{equation} \label{eq:primebd}
\le
2  \sum_{p \le
  x^{\eta_1}}  \bigg(\sum_{\substack{n \le x/p \\ (n(pn+4),P(z))=1 \\
    q|pn+4}}  1+O(z) \bigg),
\end{equation}
where we used the trivial bound $|\{p \le z \}| =O( z)$ in the last step.
Let  $$\varrho_p(d)=|\{ a \tmod d) : a(pa+4) \equiv 0 \tmod d)\}|.$$
Choosing $\theta$ so that it is sufficiently small (in terms of $\eta_2$),
it follows upon using an upper bound sieve (cf. \cite[Theorem~6.9]{cribro}), that the innermost sum on the r.h.s of
\eqref{eq:primebd} above is, for $p \le x^{\eta_1}$,
\[
 \ll \frac{x}{pq} \prod_{\substack{r \le z \\ r \text{ is prime} }}\left(1-\frac{\varrho_p(r)}{r} \right)  \ll \frac{x}{pq(\log z)^2},
\]
since $\varrho_p(r)=2$ unless $r=p$. Using the bound above in \eqref{eq:primebd} and summing over $p$ completes the proof.\end{proof}


\begin{lem} \label{lem:thirdlemma}
We have that
\[
\sum_{\substack{2x^{1-\eta_2} \le q \le x \\ q \in Q_{\delta,j}}}\sum_{\substack{n \le x \\ (n+4,P(x^{\eta_2}))=1 \\ q|n+4}} (1_{P_{\varepsilon}}\ast 1_{P_{\varepsilon}''})(n) b(n+4)  \ll   j\delta \frac{x (\log \log x)}{(\log x)^2}
\]
where the implied constant depends at most on $\eta_1$ and $\eta_2$.
\end{lem}
\begin{proof}
Note that
\begin{equation} \label{eq:firstineq}
\sum_{\substack{2x^{1-\eta_2} \le q \le x \\ q \in Q_{\delta,j}}}\sum_{\substack{n \le x \\ (n+4,P(x^{\eta_2}))=1 \\ q|n+4}} (1_{P_{\varepsilon}}\ast 1_{P_{\varepsilon}''})(n) b(n+4)  \le 2 \sum_{p \le x^{\eta_1}} \sum_{\substack{2x^{1-\eta_2} \le q \le x \\ q \in Q_{\delta,j}}} \sum_{\substack{p_1 \le x/p \\ q|pp_1+4 \\ (pp_1+4,P(x^{\eta_2}))=1}} 1
\end{equation}
Now, for a prime $p_1$ with $q|pp_1+4$ and $(pp_1+4,P(x^{\eta_2}))=1$ we have
$pp_1+4=qc\le x+4$ with $(c,P(x^{\eta_2}))=1$, and since $q \ge 2x^{1-\eta_2}$ we must have $c=1$. Hence the
two innermost sums on the r.h.s. of \eqref{eq:firstineq} are bounded by
\begin{equation} \label{eq:sumbd}
\le \sum_{\substack{q \le x \\ q\in Q_{\delta,j} \\ p|q-4 \\ (\frac{q-4}{p}, P(x^{\eta_2}))=1}} 1+O(x^{\eta_2}) \le \sum_{\substack{n \le x \\ (n(n-4)/p, \frac{P(z)}{2p})=1 \\ p|n-4}} r_{\delta}(n)+O(x^{\eta_2}+z),
\end{equation}
where
\[r_{\delta}(n):= \sum_{\substack{|z|^2=n \\
    \arg(z) \in I_{\delta}}} 1,\]
    $z=x^{\theta}$, where $\theta< \eta_2$ is sufficiently small,
    and $I_{\delta}$ is the union of
$\delta/2$-neighborhoods around the points $\frac{k\pi}{2(j+1)}$ where $0 \le k < 4(j+1)$. We will assume that $j<1/(10\delta)$ so that the intervals $I_{\delta}$ are disjoint. The case where $j>1/(10\delta)$ is similar (here one can replace $r_{\delta}(n)$ with $r(n)$ on the r.h.s. of \eqref{eq:sumbd}).
Let $\Lambda=\{\lambda_d\}_{d \le D}$ be an upper bound beta-sieve of level $D=z^{20}$ (see \cite[Section 6.4]{cribro}). Additionally, since $\Lambda$ is an upper bound sieve we have $1_{(m,\frac{P(z)}{2p})=1} =\sum_{d|(m,\frac{P(z)}{2p})} \mu(n) \le \sum_{d|(m,\frac{P(z)}{2p})} \lambda_d$, for any $m \in \mathbb N$.
Hence, we find that
\begin{equation} \label{eq:one}
\sum_{\substack{n \le x \\ (n(n-4)/p, \frac{P(z)}{2p})=1 \\ p|n-4}} r_{\delta}(n) \le \sum_{\substack{d \le D \\ d|\frac{P(z)}{2p}}} \lambda_d \sum_{\substack{a \pamod d \\ a(a-4)\equiv 0 \pamod d}}\sum_{\substack{n \le x \\ n \equiv \gamma \pamod{pd}}} r_{\delta}(n)
\end{equation}
where $\gamma=\gamma_a=d\overline d 4+p \overline p a$ (here $d\overline d \equiv 1 \tmod p)$ and $p \overline p \equiv 1 \tmod d)$) and note $(\gamma, pd)=(a,d)$.
For $b,m>0$ let
\[
\eta_b(m):=\left|\{\alpha_1,\alpha_2 \tmod m) : \alpha_1^2+\alpha_2^2 \equiv b \tmod m)\}\right|.
\]
By \cite[Proposition~A.1]{KLR} we have for $p \le x^{\eta_1}$
\begin{equation}
  \label{eq:r-epsilon-sum}
\sum_{\substack{n \le x \\ n \equiv \gamma \pamod{pd}}} r_{\delta}(n)
=4(j+1)\delta \frac{x}{(pd)^2} \eta_{\gamma}(pd)+O(x^{9/10});
\end{equation}
note that $\eta_{\gamma}(pd)=\eta_1(p)\eta_a(d)$ (see \cite[Eq'ns (A.10)-(A.12)]{KLR}). For odd square-free $d$,
let
\[g(d)= \frac{1}{d}\displaystyle \sum_{\substack{a \tmod d) \\
    a(a-4) \equiv 0 \tmod d)}} \eta_a(d)
\]
and
note $g(d)$ is a multiplicative function. Further, by \cite[(A.10)]{KLR},
\begin{equation} \label{eq:gform}
g(p)=2+\chi_4(p)-\frac{2\chi_4(p)}{p}
\end{equation}
for $p>2$, where $\chi_4$ is the non-principal character $\tmod 4)$. Using (\ref{eq:r-epsilon-sum}), we find that the right hand
side of \eqref{eq:one} equals
\begin{equation} \label{eq:asymformula}
4(j+1)\delta \frac{\eta_1(p)x}{p^2}\sum_{\substack{d \le D \\ d |\frac{P(z)}{2p}}} \frac{\lambda_d g(d)}{d}+O(x^{10/11}).
\end{equation}
Applying the  Fundamental Lemma of the Sieve (see \cite[Lemma 6.8, p. 68]{cribro}) and \eqref{eq:gform} we get
\begin{equation} \label{eq:sievebd}
 \sum_{\substack{d \le D \\ d|\frac{P(z)}{2p}}} \frac{\lambda_d g(d)}{d} \ll \prod_{p \le z} \left( 1-\frac{2+\chi_4(p)}{p}\right) \asymp \frac{1}{(\log z)^2}.
\end{equation}
Combining \eqref{eq:firstineq},\eqref{eq:sumbd},\eqref{eq:one}, \eqref{eq:asymformula} and \eqref{eq:sievebd}, as well as using that $\eta_{1}(p)/p^{2} \ll 1/p$ (see \cite[(A.11)]{KLR}), and summing over
$p$ completes the proof. \end{proof}

\begin{proof}[Proof of Proposition \ref{prop:m-size}]The first assertion of Proposition~\ref{prop:m-size} is immediate from
Lemma \ref{lem:firstlemma}.  The second assertion follows from Lemmas \ref{lem:secondlemma} and \ref{lem:thirdlemma}.  Namely, the contribution from large
$q \in [2x^{1-\eta_{2}}, x]$ is, by Lemma \ref{lem:thirdlemma},
$\ll j\delta x (\log \log x) /(\log x)^{2}$.  Using the  Lemma \ref{lem:secondlemma},
together with the estimate
\[
\sum_{q \in [x^{\eta_2},2x^{1-\eta_2 }]\cap Q_{\delta,j} } \frac{1}{q}\ll j\delta
\]
which follows from the Prime Number Theorem for Gaussian primes in sectors,
we
find that the contribution from $q \in [x^{\eta_2},x^{1-\eta_2 }] \cap Q_{\delta,j}$ is
also $\ll j\delta x (\log \log x)/(\log x)^{2}$.
\end{proof}
%


\section{Distribution of real parts}
\label{sec:distr-real-parts}

In what follows we only sketch the proof of the distribution modulo $1$
result, with fine details left to the reader.
The limit $n\rightarrow\infty$ corresponds to $|w|\uparrow 1$ on the Poincar\'{e} disk model $\Db := \{|w|<1\}\subseteq\C$,
via the map $w=W(z)=\frac{z-i}{1-iz}$ from $\Hb$ to $\Db$. We choose a large parameter $K>0$, that will be sent to infinity at the last stage,
and restrict to $\Hb\cap \{|\Re z|\le K\}$, having the effect of removing a small angular sector around $i$ of $\Db$.
It then follows that, under this restriction, all the imaginary parts $\Im \gamma(i)$ are {\em uniformly small}, and, thanks to the equidistribution
of Theorem \ref{thm:equidistributed} (that is, the angular equidistribution on $\Db$) we deduce that
the density $p$ is given by
\begin{equation}
\label{eq:f pdf Re sum}
p(x) = \sum\limits_{k\in \Z} |W'(x+k)| \cdot \frac{1}{2\pi} = \frac{1}{\pi}\sum\limits_{k\in\Z} \frac{1}{1+(x+k)^{2}},
\end{equation}
$x\in [0,1]$, with the derivative given by $W'(z) = \frac{2}{(iz-1)^{2}}$. It is then possible to sum up the series on the r.h.s. of
\eqref{eq:f pdf Re sum} to derive \eqref{eq:f pdf Re}.

\end{document}